\documentclass[12pt]{amsart}
\usepackage[colorlinks=true,citecolor=green,linkcolor=red]{hyperref}
\usepackage{amsmath}
\usepackage{verbatim}
\usepackage{amsfonts}
\usepackage{amssymb}
\usepackage{color}
\usepackage[all]{xy}
\usepackage{enumerate}
\usepackage[top=1in, bottom=1in, left=0.9in, right=0.9in]{geometry}
\usepackage{mathrsfs}

\usepackage{stmaryrd}
\usepackage{bbold}

\usepackage{cleveref}
\usepackage{soul}

\theoremstyle{theorem}
\newtheorem{thm}{Theorem}[section]
\newtheorem{theoremx}{Theorem}

\numberwithin{equation}{section}

\newtheorem{question}[thm]{Question}
\newtheorem{corollary}[thm]{Corollary}
\newtheorem{lemma}[thm]{Lemma}
\newtheorem{proposition}[thm]{Proposition}

\theoremstyle{definition}
\newtheorem{definition}[thm]{Definition}

\newtheorem{example}[thm]{Example}

\newtheorem{remark}[thm]{Remark}

\newtheoremstyle{TheoremNum}
        {8pt}{8pt}              
        {\upshape}                      
        {}                              
        {\bfseries}                     
        {.}                             
        {.5em}                             
        {\thmname{#1}\thmnote{ \bfseries #3}}
  \theoremstyle{TheoremNum}

\newcommand{\m}{\mathfrak{m}}
\newcommand{\n}{\mathfrak{n}}
\newcommand{\fpt}{{\operatorname{fpt}}}

\newcommand{\pd}{\operatorname{pd}}

\newcommand{\Spec}{\operatorname{Spec}}

\newcommand{\Hom}{\operatorname{Hom}}

\newcommand{\Ext}{\operatorname{Ext}}

\newcommand{\Ass}{\operatorname{Ass}}	
	
\newcommand{\depth}{\operatorname{depth}}

\newcommand{\Ht}{\operatorname{ht}}	
\newcommand{\BHt}{\operatorname{bight}}

\renewcommand{\leq}{\leqslant}
\renewcommand{\geq}{\geqslant}


\begin{document}

\title{Symbolic power containments in singular rings in positive characteristic}

\author{Elo\'isa Grifo}
\email{grifo@unl.edu\vspace{-0.55em}}
\address{Department of Mathematics\\University of Nebraska --- Lincoln\\Lincoln, NE 68588}
\thanks{Grifo was supported in part by NSF Grant DMS \#2001445 and \#2140355.}

\author{Linquan Ma}
\email{ma326@purdue.edu\vspace{-0.55em}}
\address{Department of Mathematics\\ Purdue University\\  West Lafayette\\ IN 47907}
\thanks{Ma was supported in part by NSF Grant DMS \#1901672, NSF FRG Grant \#1952366, and a fellowship from the Sloan Foundation.}

\author{Karl Schwede}
\email{schwede@math.utah.edu\vspace{-0.5em}}
\address{Department of Mathematics\\ University of Utah\\ Salt Lake City\\ UT 84112}
\thanks{Schwede was supported in part by NSF CAREER Grant DMS \#1252860/1501102, NSF Grant \#1801849, and a Fellowship from the Simons Foundation.}

\maketitle

\begin{abstract}
	The containment problem for symbolic and ordinary powers of ideals asks for what values of $a$ and $b$ we have $I^{(a)} \subseteq I^b$. Over a regular ring, a result by Ein--Lazarsfeld--Smith, Hochster--Huneke, and Ma--Schwede partially answers this question, but the containments it provides are not always best possible. In particular, a tighter containment conjectured by Harbourne has been shown to hold for interesting classes of ideals --- although it does not hold in general. In this paper, we develop a Fedder (respectively, Glassbrenner) type criterion for $F$-purity (respectively, strong $F$-regularity) for ideals of finite projective dimension over $F$-finite Gorenstein rings and use our criteria to extend the prime characteristic results of Grifo--Huneke to singular ambient rings. For ideals of infinite projective dimension, we prove that a variation of the containment still holds, in the spirit of work by Hochster--Huneke and Takagi.
\end{abstract}

\section{Introduction}


Given an ideal $I\subseteq R$, its $n$-th symbolic power $I^{(n)}$ consists of the elements of $R$ whose image in $R_P$ are contained in $I^nR_P$ for all $P\in \Ass(R/I)$. In particular, if $I$ is radical, $I^{(n)}$ is obtained by collecting the minimal primary components of $I^n$ (i.e., elements of $R$ that vanish generically of order $n$ at each component of $I$). It is clear that $I^n\subseteq I^{(n)}$, but the two do not coincide in general, and finding generators for $I^{(n)}$ is often a very difficult question. Comparing the two graded families $\{ I^n \}$ and $\{ I^{(n)} \}$ of ideals is a natural question, and the containment problem asks for such a comparison: when is $I^{(a)} \subseteq I^b$? When $R$ is regular, Ein--Lazarsfeld--Smith \cite{ELS}, Hochster--Huneke \cite{comparison}, and Ma--Schwede \cite{MaSchwede} gave an answer to the containment problem by showing that for all radical ideals $I\subseteq R$,
$$I^{(hn)} \subseteq I^n \text{ for all } n\geqslant 1$$
where $h$ is the \emph{big height} of $I$: the maximal height of an associated prime of $I$, which we will denote throughout by $\BHt(I)$.

These results are pioneered by work of Swanson \cite{Swanson} showing that the containment problem has an answer for $I$, meaning that for every $b$ there exists an $a$ such that $I^{(a)} \subseteq I^b$, if and only if there exists a constant $c$ such that $I^{(cn)} \subseteq I^n$ for all $n\geq 1$. However, these results do not fully settle the containment problem, and examples abound where the containments can be improved. In prime characteristic $p$ and when $R$ is regular, the stronger containment $I^{(hq-h+1)} \subseteq I^{[q]}$ holds for all $q = p^e$, while ideals $I$ with $I^{(hn-h)} \nsubseteq I^n$ can easily be constructed for any given $n$ and $h=\BHt(I)$. Harbourne \cite{Seshadri} asked what happens in the border case: is it true that for all $n \geqslant 1$ and all radical ideals $I$ in a regular ring $R$ we have
$$I^{(hn-h+1)} \subseteq I^n?$$

In a surprising twist, Dumnicki, Szemberg, and Tutaj-Gasi\'nska found a counterexample to Harbourne's conjecture \cite{counterexamples}, and soon both extensions of the original counterexample and new families of counterexamples followed \cite{HaSeFermat,MalaraSzpond,BenCounterexample,Akesseh,MalaraSzpond,RealsCounterexample}. However, the counterexamples known at the time when this paper was written all correspond to very special configurations in $\mathbb{P}^n$, and Harbourne's conjecture does hold for general sets of points in $\mathbb{P}^2$ \cite{BoH}, $\mathbb{P}^3$ \cite{Dumnicki2015}, and other special configurations of points, such as the star configurations \cite{HaHu}. In prime characteristic $p$, Grifo and Huneke \cite{GrifoHuneke} showed that Harbourne's conjecture holds for ideals defining $F$-pure rings, making use of Fedder's Criterion \cite{Fedder}. For the subclass of ideals defining strongly $F$-regular rings, they showed that one can replace the big height $h \geqslant 2$ with $h-1$. This yields equality of symbolic and ordinary powers when $h=2$.

When the ambient ring $R$ is not regular, much less is known. Swanson's work \cite{Swanson} has inspired a search for rings with the uniform symbolic topologies property, that is, rings with a constant $c$ not depending on the prime $P$ such that $P^{(cn)} \subseteq P^n$ for all $n\geq 1$. It is conjectured that a uniform $c$ should exist in great generality, and while the question remains open in general, it has been solved in many interesting cases \cite{RobertHH,RobertNormal,RobertUSTP,JavierDanielDiagFreg,HKV,HunekeKatzAbelian}.

The main goal of this paper is to extend the aforementioned results of Grifo and Huneke to singular ambient rings. That is, instead of studying all radical ideals $I\subseteq R$, we focus on the case that $R/I$ is $F$-pure or strongly $F$-regular (but $R$ need not be regular). Our first main result is the following:

\begin{theoremx}[Theorem \ref{thm.MainFinitePdim}]
Let $R$ be an $F$-finite Gorenstein ring of prime characteristic $p$ and $Q\subseteq R$ be an ideal of finite projective dimension with big height $h$. Then
\begin{enumerate}
  \item If $R/Q$ is $F$-pure, then $Q^{(hn-h+1)}\subseteq Q^n$ for all $n \geqslant 1$.
  \item If $R/Q$ is strongly $F$-regular and $h\geq 2$, then $Q^{((h-1)(n-1)+1)}\subseteq Q^{n}$ for all $n \geqslant 1$.
\end{enumerate}
\end{theoremx}

In a regular ring, every ideal has finite projective dimension, hence the theorem above extends \cite[Theorem 3.2 and Theorem 4.1]{GrifoHuneke}.
As a central ingredient of the proof of this result, we extend Fedder's Criterion \cite{Fedder} for $F$-purity and Glassbrenner's Criterion \cite{Glassbrenner} for strong $F$-regularity to ideals of finite projective dimension in an $F$-finite Gorenstein local ring.  We believe these results will be of independent interest.

\begin{theoremx}[Corollaries \ref{cor.Fedder} and \ref{cor.Glassbrenner}]
Let $(R,\m)$ be a local $F$-finite Gorenstein ring and let $Q\subseteq R$ be an ideal such that $\pd(R/Q)<\infty$.
\begin{enumerate}
  \item $R/Q$ is $F$-pure if and only if $\left( Q^{[p^e]}:Q \right) \nsubseteq I_e(\m)$ for some (equivalently every) $e > 0$.\footnote{Here $I_e(\m)$ is the set of elements $r \in R$ such that $\phi(F^e_* r) \in \m$ for every $\phi \in \Hom_R(F^e_* R, R)$.}
  \item $R/Q$ is strongly $F$-regular if and only if for any $c \notin Q$ there exists $e$ such that $c(I_e(Q):Q) \nsubseteq I_e(\m)$
\end{enumerate}
\end{theoremx}

To show these results we generalize, to quotients of Gorenstein rings by ideals of finite projective dimension, Fedder's key lemma on lifting $p^{-1}$-linear maps in quotients of regular rings, {\itshape c.f.} \cite[Corollaries to Lemma 1.6]{Fedder}.
\begin{theoremx}[Lemma \ref{lem.LiftMap}]
	Let $R$ be a local $F$-finite Gorenstein ring and let $Q\subseteq R$ be an ideal of finite projective dimension. Then any map $\phi\in \Hom_{R/Q}(F^e_*(R/Q), R/Q)$ lifts to a map $\widetilde{\phi}\in \Hom_R(F^e_*R, R)$.
\end{theoremx}

For an arbitrary radical ideal $I\subseteq R$ (i.e., not necessarily having finite projective dimension), the containment $I^{(hn)} \subseteq I^n$ may easily fail for $h=\BHt(I)$. However, Takagi \cite[Theorem 4.6]{TakagiFormulas} showed that we can always fix the failure of this containment by multiplying by a power of the Jacobian ideal, tightening a result of Hochster and Huneke \cite[Theorem 3.7 (b)]{comparison}. In the same spirit, we show that we can fix the failure of Harbourne-like containments by multiplying by a power of the Jacobian ideal.

\begin{theoremx}[Theorem \ref{thm.MainInfPdim}]
Let $R$ be a geometrically reduced equidimensional $k$-algebra finitely generated over a field $k$ of prime characteristic $p$. Let $Q\subseteq R$ be an ideal and let $J=J(R/k)$ be the Jacobian ideal of $R$. Then
\begin{enumerate}
  \item If $R/Q$ is $F$-pure and $Q$ has big height $h$, then $J^n Q^{(hn-h+1)} \subseteq Q^n$ for all $n \geq 1$. 
  \item If $R/Q$ is strongly $F$-regular and $h\geq 2$ is at least the largest minimal number of generators of $QR_P$ for all $P\in\Ass(R/Q)$, then $J^{2n-2} Q^{((h-1)(n-1)+1)} \subseteq Q^n$ for all $n \geqslant 1$.
\end{enumerate}
\end{theoremx}

Takagi's result \cite{TakagiFormulas} relies on a refinement of the subadditivity formula for test ideals. Our method is different and it relates directly to the liftablity of $p^{-1}$-linear maps; while not every such map can be lifted from $R/Q$ to $R$, we show that under various assumptions, every $p^{-1}$-linear map can be lifted after multiplication by elements in the Jacobian ideal (Lemma \ref{lem.LiftingMapJacobian}).

Finally, in Section 5, we extend a containment result by Takagi and Yoshida \cite[Remark 3.4]{TakagiYoshida} involving $F$-pure thresholds to non-regular rings (see Theorems \ref{fpt finite pdim} and \ref{fpt infinite pdim}), and in Section 6, we provide some examples showing the sharpness of our result. Throughout this article, all rings are commutative, Noetherian, with unity, and have positive prime characteristic $p$.

\section{Preliminaries}

In this section, we recall some definitions and collect some basic facts about the Frobenius non-splitting ideals $I_e(-)$. We use $F^e_*R$ to denote the Frobenius pushforward of $R$, i.e., the target of the natural Frobenius map $R\to F_*^eR$. We say $R$ is {\it $F$-finite} if $F^e_*R$ is a finite $R$-module.

\begin{definition}[{\itshape c.f.} \cite{AE}]
Suppose $R$ is an $F$-finite ring.  Let $Q\subseteq R$ be an ideal and let $\varphi \in \Hom_R(F^e_*R, R)$. We define
\[
I_e^\varphi(Q) = \left\lbrace r \in R : \varphi(F^e_*(rR)) \subseteq Q \right\rbrace.
\]
Even without a fixed $\varphi$, we define
\[
I_e(Q) = \left\lbrace r \in R : \varphi(F^e_*r) \in Q \textrm{ for all } \varphi \in \Hom_R(F^e_*R, R) \right\rbrace.
\]
\end{definition}

The ideals $I_e$ can be viewed as a different way to generalize $Q^{[p^e]}$ from the regular case to the singular case.  Indeed, it is straightforward to see that
\[
Q^{[p^e]} \subseteq I_e(Q) \subseteq I_e^\varphi(Q).
\]
On the other hand, by \cite[Lemma 1.6]{Fedder}, it follows that $Q^{[p^e]} = I_e(Q)$ if $R$ is regular.  We will generalize this result to radical ideals of finite projective dimension at the end of this section in Lemma \ref{lem.AssPrimeFinitePdim}.

\begin{remark}[The Gorenstein case]
When $R$ is local $F$-finite and Gorenstein, $\Hom_R(F^e_*R, R)\cong (F^e_*R)\cdot \Phi_e$ where $\Phi_e$ is the Grothendieck trace map. Therefore, in this case
\[
I_e(Q) = \left\lbrace r \in R : \Phi_e(F^e_*(rR)) \subseteq Q \right\rbrace = I_e^{\Phi}(Q).
\]
\end{remark}

We next recall the definitions of $F$-pure and strongly $F$-regular rings.

\begin{definition}
Let $R$ be an $F$-finite ring. We say $R$ is \emph{$F$-pure} if the natural map $R\to F_*^eR$ splits as a map of $R$-modules for one (or equivalently, all) $e>0$. We say $R$ is \emph{strongly $F$-regular}, if for every $c$ not in any minimal prime of $R$, there exists $e>0$ such that the map $R\to F^e_*R$ sending $1$ to $F^e_*c$ splits as a map of $R$-modules.
\end{definition}

\begin{lemma}
\label{lem.MapRestrict}
Let $R$ be a local $F$-finite Gorenstein ring and consider an ideal $Q\subseteq R$. The map $\Phi_e(F^e_*r\cdot-)\in \Hom_R(F^e_*R, R)$ induces a map in $\Hom_{R/Q}(F^e_*(R/Q), R/Q)$ if and only if $r\in (I_e(Q): Q)$.
\end{lemma}
\begin{proof}
The map needs to send $F^e_*Q$ to $Q$, which means $\Phi_e(F^e_*(rQ))\subseteq Q$. This happens if and only if $rQ\subseteq I_e(Q)$ by definition.
\end{proof}

\begin{lemma}
\label{lem.PhiI_e}
	Let $I$, $J$ and $Q$ be ideals in $R$ and $\varphi \in \Hom_R(F^e_*R,R)$. Then for all $q=p^e$,
	$$\varphi \left(F^e_* (I_e(Q) I^{[q]} : J^{[q]}) \right)\subseteq\varphi \left(F^e_* (I^\varphi_e(Q) I^{[q]} : J^{[q]}) \right) \subseteq  \left( Q I : J \right).$$
\end{lemma}

\begin{proof}
	The first inclusion is trivial since $I_e(Q)\subseteq I_e^\varphi(Q)$. Let $r \in \left( I^\varphi_e(Q) I^{[q]} : J^{[q]} \right)$. Then
\[\varphi(F^e_*r) J = \varphi(F^e_*(r J^{[q]})) \subseteq \varphi\left(F^e_*(I^\varphi_e(Q) I^{[q]})\right) = \varphi\left( F^e_*(I^\varphi_e(Q)) \right) I \subseteq QI. \qedhere \]
\end{proof}

In the regular setting, Frobenius powers preserve associated primes, an idea which plays an important role in \cite{comparison}. This is no longer true in general. However, if we replace Frobenius powers with $I_e(-)$ ideals, then the analogous statement holds.

\begin{lemma}
\label{lem.AssIeQisQ}
For any ideal $Q\subseteq R$, we have $\Ass(R/I_e(Q)) \subseteq \Ass(R/Q)$ for all $e > 0$. Moreover, if $R$ is $F$-pure then $\Ass(R/I_e(Q)) = \Ass(R/Q)$.
\end{lemma}
\begin{proof}
First, note that since $Q^{[q]}\subseteq I_e(Q)$ for $q=p^e$, associated primes of $I_e(Q)$ contain $Q^{[q]}$ and hence contain $Q$. Now given a prime $P \supseteq Q$, if $P \notin \Ass(R/Q)$ then $(Q:P) = Q$. Hence by \cite[Lemma 2.3 (7)]{PolstraSmirnov},
$$\left( I_e(Q) : \left( P \right)^{[q]} \right) = I_e (Q : P) = I_e(Q).$$
This implies $P\notin \Ass(R/I_e(Q))$ and hence $\Ass(R/I_e(Q)) \subseteq \Ass(R/Q)$.

Conversely, if $P\supseteq Q$ and $P\notin \Ass(R/I_e(Q))$, then we have
$$I_e(Q) = \left( I_e (Q) : P^{[q]} \right) = I_e \left( Q : P \right).$$
By \cite[Corollary 2.4]{PolstraSmirnov}, this implies $\left( Q : P \right) = Q$ when $R$ is $F$-pure, and hence $P \notin \Ass(R/Q)$. Thus $\Ass(R/I_e(Q)) = \Ass(R/Q)$ when $R$ is $F$-pure.
\end{proof}

\begin{definition}
Let $I\subseteq R$ be an ideal and $W$ be a multiplicative system. We denote by $I^W$ those elements of $R$ whose images in $W^{-1}R$ are contained in $W^{-1}I$.
\end{definition}

\begin{lemma}
\label{lem.AssPrimeFinitePdim}
Let $Q\subseteq R$ be an ideal such that $\pd (R/Q) < \infty$. Then $(Q^{[q]})^W=Q^{[q]}$ for all $q=p^e$, where $W$ denotes the multiplicative system $R-\cup_{P \in \Ass(R/Q)} P$. Moreover, if $Q$ is radical then we also have $I_e(Q)=Q^{[q]}$ for all $q=p^e$. 
\end{lemma}

\begin{proof}
By \cite[Lemma 2.2]{comparison}, $\Ass(R/Q)=\Ass(R/Q^{[q]})$ since $\pd (R/Q)<\infty$. The first assertion follows immediately from this.

For the second assertion, notice that $Q^{[q]}\subseteq I_e(Q)$ and $\Ass(R/Q^{[q]}) = \Ass(R/Q)$ by \cite[Lemma 2.2]{comparison}, thus to show $Q^{[q]}=I_e(Q)$ it is enough to check that $Q^{[q]}=I_e(Q)$ after localizing at each $P\in\Ass(R/Q^{[q]})= \Ass(R/Q)$. But since $\pd(R/Q)<\infty$ and $Q$ is radical, $R_P$ is regular and so $I_e(Q)R_P=I_e(QR_P)=Q^{[q]}R_P$.  
\end{proof}

We also need the following lemma which is a version of \cite[Lemma 2.6]{GrifoHuneke} in the non-regular setting, where again we replace Frobenius powers with $I_e(-)$ ideals.

\begin{lemma}\label{q lemma for I_e}
Let $Q\subseteq R$ be an ideal and let $h$ be the largest minimal number of generators of $QR_P$ for all $P\in\Ass(R/Q)$. Then for every $n\geq 1$ and $q=p^e$ we have
$$Q^{(q(h+n-1)-h+1)} \subseteq I_e(Q^{(n)}).$$
\end{lemma}
\begin{proof}
By Lemma \ref{lem.AssIeQisQ}, $\Ass \left( R/ I_e(Q^{(n)}) \right) \subseteq \Ass \left( R / Q^{(n)} \right)$, and $\Ass \left( R / Q^{(n)} \right) = \Ass \left( R / Q \right)$ by definition of $Q^{(n)}$. Thus it is enough to show that the inclusion holds after localizing at associated primes of $Q$. But after localizing at $P \in \Ass(R/Q)$, ordinary and symbolic powers of $Q$ are the same and by assumption we have that $Q$ is now generated by $h$ elements. Thus by \cite[Lemma 2.4]{comparison} (see also \cite[Lemma 2.5]{GrifoHuneke}), we have
\[Q^{q(h+n-1)-h+1}R_P\subseteq (Q^n)^{[q]}R_P=(Q^{(n)})^{[q]}R_P\subseteq I_e(Q^{(n)}R_P)=I_e(Q^{(n)})R_P. \qedhere\]
\end{proof}

\section{Ideals of finite projective dimension}

The goal of this section is to give Fedder-type criteria for $F$-purity and strong $F$-regularity of quotients of $F$-finite Gorenstein rings and to apply these to obtain symbolic power containments similar to \cite{GrifoHuneke}. 

\subsection{A generalization of Fedder's criterion}

The following lemma on lifting Cartier maps is the key ingredient we will use to extend Fedder's criterion to singular ambient rings. We believe this lemma is also of independent interest.

\begin{lemma}
\label{lem.LiftMap}
Let $R$ be a local $F$-finite Gorenstein ring and let $Q\subseteq R$ be an ideal of finite projective dimension. Then any map $\phi\in \Hom_{R/Q}(F^e_*(R/Q), R/Q)$ lifts to a map $\widetilde{\phi}\in \Hom_R(F^e_*R, R)$.
\end{lemma}
\begin{proof}
Fix $\phi: F^e_*(R/Q)\to R/Q$. Consider its Grothendieck dual, and note that since $R$ is local Gorenstein, $\omega_R^\bullet\cong R$:
$$\phi^\vee: \mathbf{R}\Hom_R(R/Q, R)\to \mathbf{R}\Hom_R(F^e_*(R/Q), R).$$
Given a finite free resolution of $R/Q$,
$$G_\bullet: 0\to R^{n_s}\to R^{n_{s-1}}\to\cdots\to R^{n_1}\to R\to 0,$$
we have $\mathbf{R}\Hom_R(R/Q, R)\cong G_\bullet^*$, and by duality $\mathbf{R}\Hom_R(F^e_*(R/Q), R)\cong F^e_*(G_\bullet^*)$.  Since $G_\bullet^*$ is a finite complex of free $R$-modules, $\phi^\vee$ can be realized as an honest map of chain complexes: $G_\bullet^*\to F^e_*(G_\bullet^*)$. That is:
\[\xymatrix{
0 & {R^{n_s}}^*\ar[l]\ar[d] & {R^{n_{s-1}}}^* \ar[l]\ar[d] & \cdots\ar[l] & {R^{n_1}}^*\ar[l]\ar[d] & R^*\ar[l]\ar[d] & 0\ar[l]  \\
0 & F^e_*({R^{n_s}}^*)\ar[l] & F^e_*({R^{n_{s-1}}}^*)\ar[l] & \cdots \ar[l] & F^e_*({R^{n_1}}^*)\ar[l] & F^e_*(R^*)\ar[l] & 0\ar[l]
}.
\]
Now we apply $\Hom_R(-, R)$ to this diagram, we obtain:
\[\xymatrix{
0 \ar[r] & {R^{n_s}}\ar[r] & {R^{n_{s-1}}} \ar[r] & \cdots\ar[r] & {R^{n_1}}\ar[r] & R \ar@{.>}[r] & R/Q \ar@{.>}[r] & 0  \\
0 \ar[r] & F^e_*{R^{n_s}}\ar[r] \ar[u] & F^e_*{R^{n_{s-1}}}\ar[r] \ar[u] & \cdots \ar[r] & F^e_*{R^{n_1}} \ar[r]\ar[u] & F^e_*R \ar@{.>}[r]\ar[u]^{\widetilde{\phi}} & F^e_*(R/Q) \ar[u]^{\phi} \ar@{.>}[r] & 0
}.
\]

We claim that this diagram represents $\phi$ viewed as a map in $\Hom_{D(R)}(F^e_*(R/Q), R/Q)$. Since the first row is a complex of finite free modules, applying $\Hom_R(-, R)$ is the same as applying $\mathbf{R}\Hom_R(-, R)$. For the second row, note that by duality $$\mathbf{R}\Hom_R(F^e_*R, R)\cong F^e_*R\cong \Hom_R(F^e_*R, R)\cong h^0(\mathbf{R}\Hom_R(F^e_*R, R))$$ and hence $F^e_*R$ is acyclic when applying $\mathbf{R}\Hom_R(-, R)$. Therefore, applying $\Hom_R(-, R)$ is also the same as applying $\mathbf{R}\Hom_R(-, R)$ to the second row (i.e., they represent the same object in the derived category). This shows the diagram represents $\phi^{\vee\vee}=\phi$ in the derived category. In particular, $\phi$ is exactly the induced map on the zeroth cohomology of this diagram. Therefore the map $\widetilde{\phi}$ is a lift of $\phi$ to $\Hom_R(F^e_*R, R)$ as desired.
\end{proof}

\begin{remark}
\label{rmk.Sri}
Srikanth Iyengar communicated to us an alternative proof of Lemma \ref{lem.LiftMap}, which we sketch here. Since $R$ is $F$-finite Gorenstein and local, $\Hom_R(F_*^eR, R)\cong \mathbf{R}\Hom_R(F^e_*R, R)$\footnote{In fact, it is well-known that under mild assumptions, the property that $\Hom_R(F_*^eR, R)\cong \mathbf{R}\Hom_R(F^e_*R, R)$ for all $e$ characterizes Gorenstein rings, see \cite{Herzog1974}.} and thus $$\Hom_R(F^e_*R, R)\twoheadrightarrow \Hom_R(F^e_*R, R)\otimes_R(R/Q)\cong h^0(\mathbf{R}\Hom_R(F^e_*R, R)\otimes^\mathbf{L}_R(R/Q)).$$
Since $\pd(R/Q)<\infty$, we know that
$$\mathbf{R}\Hom_R(F^e_*R, R)\otimes^\mathbf{L}_R(R/Q)\cong \mathbf{R}\Hom_R(F^e_*R, R/Q)\cong \mathbf{R}\Hom_{R/Q}(F^e_*R\otimes^\mathbf{L}_R(R/Q), R/Q).$$
By \cite[Th\'{e}or\`{e}me (I.7)]{P-S}, $F^e_*R\otimes^\mathbf{L}_R(R/Q)\cong F^e_*R\otimes_R(R/Q)$ since $\pd(R/Q)<\infty$. Now as $F^e_*R\otimes_R(R/Q)\twoheadrightarrow F^e_*(R/Q)$, we have
\[
\begin{array}{rl}
& \Hom_{R/Q}(F^e_*(R/Q), R/Q)\\
\hookrightarrow& \Hom_{R/Q}(F^e_*R\otimes_R(R/Q), R/Q)\\
=& h^0(\mathbf{R}\Hom_{R/Q}(F^e_*R\otimes^\mathbf{L}_R(R/Q), R/Q)).
\end{array}
\]
Putting all these together we find that
$$\Hom_R(F_*^eR, R)\twoheadrightarrow h^0(\mathbf{R}\Hom_R(F^e_*R, R)\otimes^\mathbf{L}_R(R/Q))\hookleftarrow \Hom_{R/Q}(F^e_*(R/Q), R/Q),$$
which is precisely saying that any $\phi\in \Hom_{R/Q}(F^e_*(R/Q), R/Q)$ comes from a map $\widetilde{\phi}\in \Hom_R(F_*^eR, R)$ under the natural restriction map.
\end{remark}

\begin{thm}
\label{thm.LiftFpure}
Let $R$ be a local $F$-finite Gorenstein ring and let $Q\subseteq R$ be an ideal such that $\pd(R/Q)<\infty$.
\begin{enumerate}
  \item If $R/Q$ is $F$-pure, then $\Phi_e(F^e_*(I_e(Q):Q))=R$. In particular, $R$ is $F$-pure.
  \item If $R/Q$ is strongly $F$-regular, then for any $c \notin Q$ there exists $e_0$ such that for all $e>e_0$, $\Phi_e\left(F^e_*(c(I_e(Q):Q))\right)=R$. In particular, $R$ is strongly $F$-regular.
\end{enumerate}
\end{thm}

\begin{proof}The proof of the two cases are similar and both follow from Lemma \ref{lem.LiftMap}.
\begin{enumerate}[(1)]
	\item By Lemma \ref{lem.LiftMap}, any Frobenius splitting $\phi\in \Hom_{R/Q}(F^e_*(R/Q), R/Q)$ can be lifted to some $\widetilde{\phi}\in \Hom_R(F^e_*R, R)$. Since $R$ is local Gorenstein, $\widetilde{\phi}=\Phi_e(F^e_*r\cdot-)$ for some $r\in \left( I_e(Q):Q \right)$ by Lemma \ref{lem.MapRestrict}. Therefore $\Phi_e(F^e_*(I_e(Q):Q))=R$ and it follows that $R$ is $F$-pure.

\item Since $R/Q$ is strongly $F$-regular, there exists $e_0$ such that for all $e>e_0$, there exists $\phi\in \Hom_{R/Q}(F^e_*(R/Q), R/Q)$ such that $\phi(F^e_*c)=1$. We can lift $\phi$ to $\widetilde{\phi}\in \Hom_R(F^e_*R, R)$ by Lemma \ref{lem.LiftMap}. Since $R$ is local Gorenstein, $\widetilde{\phi}=\Phi_e(F^e_*r\cdot-)$ for some $r\in I_e(Q):Q$ by Lemma \ref{lem.MapRestrict}. Therefore $\Phi_e\left(F^e_*(c(I_e(Q):Q))\right)=R$. Since $R/Q$ is local and strongly $F$-regular, $Q$ is prime ideal. Since $\pd(R/Q)<\infty$, $R_Q$ is regular. Therefore we can pick $c\notin Q$ that is a test element for $R$, and $\Phi_e\left(F^e_*(c(I_e(Q):Q))\right)=R$ implies $R\to F^e_*R$ sending $1$ to $F^e_*c$ splits. Thus $R$ is strongly $F$-regular. \qedhere
\end{enumerate}
\end{proof}

\begin{corollary}[Fedder's criterion in singular rings]
\label{cor.Fedder}
Let $(R,\m)$ be a local $F$-finite Gorenstein ring and $Q\subseteq R$ be an ideal. Consider the following three conditions:
\begin{enumerate}
  \item $\left( Q^{[p^e]}:Q \right) \nsubseteq I_e(\m)$.
  \item $\left( I_e(Q):Q \right) \nsubseteq I_e(\m)$.
  \item $R/Q$ is $F$-pure.
\end{enumerate}
We have $(1)\Rightarrow (2)\Rightarrow (3)$, and moreover $(3)\Rightarrow(1)$ if $\pd(R/Q)<\infty$.
\end{corollary}
\begin{proof}
$(1)\Rightarrow(2)$ since $Q^{[p^e]}\subseteq I_e(Q)$, $(2)\Rightarrow(3)$ by Lemma \ref{lem.MapRestrict}. Finally, when $\pd(R/Q)<\infty$ and $R/Q$ is $F$-pure, we know that $\Phi_e(F^e_*(I_e(Q):Q))=R$ by Theorem \ref{thm.LiftFpure}, so $(I_e(Q):Q)\nsubseteq I_e(\m)$ and hence $(Q^{[p^e]}:Q) \nsubseteq I_e(\m)$ by Lemma \ref{lem.AssPrimeFinitePdim}.
\end{proof}

Note that when $R$ is regular, every ideal has finite projective dimension and $I_e(\m)=\m^{[p^e]}$, so in that setting Corollary \ref{cor.Fedder} is precisely the classical Fedder's criterion \cite{Fedder}. By the same proof, we also have the following, which extends \cite{Glassbrenner}.

\begin{corollary}[Glassbrenner's criterion in singular rings]
\label{cor.Glassbrenner}
Let $(R,\m)$ be a local $F$-finite Gorenstein ring and $Q\subseteq R$ be an ideal. Consider the following three conditions:
\begin{enumerate}
  \item For any $c$ not in any minimal prime of $Q$, there exists $e$ such that $c(Q^{[p^e]}:Q) \nsubseteq I_e(\m)$.
  \item For any $c$ not in any minimal prime of $Q$, there exists $e$ such that $c(I_e(Q):Q) \nsubseteq I_e(\m)$.
  \item $R/Q$ is strongly $F$-regular.
\end{enumerate}
We have $(1)\Rightarrow (2)\Rightarrow (3)$, and moreover $(3)\Rightarrow(1)$ if $\pd(R/Q)<\infty$.
\end{corollary}

\begin{remark}
The above results can also be viewed as a strong sort of inversion of adjunction of $F$-purity (in arbitrary codimension).  It would be natural to ask if analogous results held for log canonical singularities in characteristic zero.  The corresponding statements include statements of the following form.  Suppose that $R$ is a Gorenstein normal domain of finite type over $\mathbb{C}$.
\begin{enumerate}
\item If $R/Q$ has finite projective dimension and is semi-log canonical, then does there exist a $\mathbb{Q}$-divisor on $\Spec R$ such that $(R, \Delta)$ is log canonical with log canonical center $V(Q)$?
\item If $R/Q$ has finite projective dimension and is log terminal, then does there exist a $\mathbb{Q}$-divisor on $\Spec R$ such that $(R, \Delta)$ is log canonical with minimal log canonical center at $V(Q)$?
\end{enumerate}
\end{remark}

\begin{remark}
One cannot expect Lemma \ref{lem.LiftMap} and Theorem \ref{thm.LiftFpure} to hold without the Gorenstein hypothesis on $R$. In \cite{AnuragNotDeform}, Singh constructed examples of Cohen-Macaulay non-Gorenstein $R$ such that $R/(x)$ is strongly $F$-regular for a nonzerodivisor $x$ (so clearly $\pd R/(x)<\infty$) but $R$ is not even $F$-pure. In particular, no Frobenius splitting of $R/(x)$ lifts to $R$.
\end{remark}

\subsection{Applications to symbolic powers}
If $W$ is the multiplicative system $R-\cup_{P\in\Ass(R/I)}P$ (so $W$ consists of all nonzerodivisor on $R/I$), note that $(I^n)^W = I^{(n)}$ is the $n$-th symbolic power of $I$.

\begin{lemma}
\label{lem.ContainmentContraction}
Let $Q\subseteq R$ be an ideal and let $n\geq 2$ be a positive integer. Then for all $q = p^e$ sufficiently large,
$$\left( Q^{(hn-h+1)} \right)^{\left\lfloor \frac{q-1}{n-1} \right\rfloor} \subseteq (Q^{[q]})^W,$$
where $W$ denotes the multiplicative system $R-\cup_{P \in \Ass(R/Q)} P$ and $h$ is the largest analytic spread of $QR_P$ for all $P\in\Ass(R/Q)$.
\end{lemma}

\begin{proof}
It suffices to prove this after localizing at $P$ for each $P\in\Ass(R/Q)$. Therefore we may assume that the left hand side is the ordinary power of $Q$ and the right hand side is $Q^{[q]}$, and we reduce to show a containment in $R_P$. Then it is enough to check such a containment in a faithfully flat extension of $R_P$. Thus we may assume that the residue field of $R_P$ is infinite, and therefore we may assume that $Q$ is integral over an ideal $I$ generated by $h$ elements. Now for all large $q$ we write
	$$q-1 = a (n-1) + r,$$
	where $a, r \geqslant 0$ and $r \leqslant n-2$. Then
	$$(hn-h+1) \left\lfloor \frac{q-1}{n-1} \right\rfloor = (hn-h+1)a = h(q-1-r)+a = hq-h+a-rh.$$
	Thus we have
$$Q^{(hn-h+1)\left\lfloor \frac{q-1}{n-1} \right\rfloor}=Q^{hq-h+(a-rh)}.$$
Since $Q$ is integral over $I$, there exists $s$ such that for all $t\geq 1$,
	$$Q^sQ^t=Q^sI^t\subseteq I^t.$$
Now for $q\gg0$ such that $a=\lfloor\frac{q-1}{n-1}\rfloor\geq rh+s+1$ we have
$$Q^{hq-h+(a-rh)}\subseteq Q^sQ^{hq-h+1}\subseteq I^{hq-h+1}\subseteq I^{[q]}\subseteq Q^{[q]},$$ where we use the Pigeonhole Principle for $I^{hq-h+1}\subseteq I^{[q]}$ (since $I$ is generated by $h$ elements).
This finishes the proof.
\end{proof}

\begin{lemma}
\label{lem.I_econtainment1}
	Let $Q\subseteq R$ be an ideal with $\pd(R/Q)<\infty$. Let $h$ be the largest analytic spread of $QR_P$ for all $P\in\Ass(R/Q)$. Then for all $n \geqslant 1$,
	$$\left( I : Q \right) \subseteq \left( I  \left( Q^{n-1} \right)^{[q]} : \left( Q^{(hn-h+1)} \right)^{[q]} \right)$$
	for all $q = p^e \gg 0$ and all ideal $I\subseteq R$.
\end{lemma}
\begin{proof}
	If $n=1$, the desired inclusion comes down to $\left( I : Q \right) \subseteq \left( I : Q^{[q]} \right)$ which is trivial. Now we fix $n\geq 2$. Let $t \in \left( I : Q \right)$. Then
	\[
        t \left( Q^{(hn-h+1)} \right)^{[q]} \subseteq (tQ) \cdot \left( Q^{(hn-h+1)} \right)^{q-1} \subseteq I \, \left( Q^{(hn-h+1)} \right)^{q-1}.
    \]
	The result will follow if we show that
	$$\left( Q^{(hn-h+1)} \right)^{q-1} \subseteq \left( Q^{n-1} \right)^{[q]} = \left( Q^{[q]} \right)^{n-1}.$$
Therefore it is enough to prove
		$$\left( Q^{(hn-h+1)} \right)^{\left\lfloor \frac{q-1}{n-1} \right\rfloor} \subseteq  Q^{[q]}.$$
This follows by Lemma \ref{lem.ContainmentContraction} and Lemma \ref{lem.AssPrimeFinitePdim}.
\end{proof}

The next lemma will be crucial to handle the strongly $F$-regular case.
\begin{lemma}
\label{lem.I_econtainment2}
Let $Q\subseteq R$ be an ideal and let $h\geq 2$ be the largest minimal number of generators of $QR_P$ for all $P \in \Ass(R/Q)$. Then for all $n$ and for all $e \gg 0$, we have
$$\left( Q^{n+h-1} : Q^{(n+h-1)} \right) \left( I_e(Q) : Q \right) \subseteq \left( I_e(Q) I_e \left( Q^{(n)} \right) : \left( Q^{(n+h-1)}  \right)^{[q]} \right).$$
In particular, if $Q$ is radical and $\pd(R/Q)<\infty$, then we have
$$\left( Q^{n+h-1} : Q^{(n+h-1)} \right) \left( I_e(Q) : Q \right) \subseteq \left( Q^{[q]} I_e \left( Q^{(n)} \right) : \left( Q^{(n+h-1)}  \right)^{[q]} \right).$$
\end{lemma}
\begin{proof}
Let $s \in \left( I_e(Q) : Q \right)$ and $t \in \left( Q^{n+h-1} : Q^{(n+h-1)} \right)$. Then
\begin{align*}
st \left( Q^{(n+h-1)}  \right)^{[q]} & \subseteq st \left( Q^{(n+h-1)}  \right)^{q} \\
& \subseteq s \left( t Q^{(n+h-1)} \right) \left( Q^{(n+h-1)} \right)^{q-1} \\
& \subseteq (sQ) Q^{n+h-2} \left( Q^{(n+h-1)} \right)^{q-1} \\
& \subseteq I_e(Q) Q^{n+h-2} \left( Q^{(n+h-1)} \right)^{q-1} \\
& \subseteq I_e(Q) Q^{((n+h-1)(q-1) + n + h - 2)}.
\end{align*}
Now note that if $h \geqslant 2$, then
\begin{align*}
(n+h-1)(q-1) + n + h - 2 & = (n+h-1)q -1 \\
& \geqslant (n+h-1)q - h + 1
\end{align*}
so by Lemma \ref{q lemma for I_e},
$$Q^{((n+h-1)(q-1) + n + h - 2)} \subseteq I_e(Q^{(n)}).$$
This shows that
$$st \left( Q^{(n+h-1)}  \right)^{[q]} \subseteq I_e(Q) I_e \left( Q^{(n)} \right),$$
as desired. The last assertion follows from Lemma \ref{lem.AssPrimeFinitePdim}.
\end{proof}

Now we state and prove our main result on symbolic power containments for ideals of finite projective dimension.

\begin{thm}
\label{thm.MainFinitePdim}
Let $R$ be an $F$-finite Gorenstein ring and $Q\subseteq R$ be an ideal with $\pd(R/Q)<\infty$ and big height $h$. Then we have
\begin{enumerate}
  \item If $R/Q$ is $F$-pure, then $Q^{(hn-h+1)}\subseteq Q^n$ for all $n \geqslant 1$.
  \item If $R/Q$ is strongly $F$-regular and $h\geq 2$, then $Q^{((h-1)(n-1)+1)}\subseteq Q^n$ for all $n \geqslant 1$.
\end{enumerate}
\end{thm}

\begin{remark}\label{remark fin pd analytic spread}
	Under the hypothesis of Theorem \ref{thm.MainFinitePdim}, $R_P$ is regular for all $P \in \Ass(R/Q)$, since $\pd (R/Q) < \infty$ and $Q$ is radical. Thus the largest minimal number of generators of $QR_P$ and the largest analytic spread of $Q R_P$ for all $P\in\Ass(R/Q)$ are both the same as $\BHt(Q)$.
\end{remark}

\begin{proof}[Proof of Theorem \ref{thm.MainFinitePdim}]
Since it is enough to check the containment after localizing at each maximal ideal of $R$, we may assume $R$ is local.
\begin{enumerate}
  \item Combining Theorem \ref{thm.LiftFpure}, Lemma \ref{lem.I_econtainment1} and Lemma \ref{lem.PhiI_e}, we have
\begin{align*}
R & \subseteq  \Phi_e(F^e_*(I_e(Q):Q)) & \textrm{ by Theorem \ref{thm.LiftFpure}} \\
& \subseteq   \Phi_e \left( F^e_*  \left( I_e(Q) \left( Q^{n-1} \right)^{[q]} : \left( Q^{(hn-h+1)} \right)^{[q]} \right) \right) & \textrm{ by Lemma \ref{lem.I_econtainment1}} \\
& \subseteq  \left( Q Q^{n-1} : Q^{(hn-h+1)} \right) & \textrm{ by Lemma  \ref{lem.PhiI_e}}.
\end{align*}

Therefore, $Q^{(hn-h+1)}\subseteq Q^n$, as desired.
  \item There exists an element $c \in \left( Q^{n+h-1} : Q^{(n+h-1)} \right)$ not in any minimal prime of $Q$. Thus combining Theorem \ref{thm.LiftFpure}, Lemma \ref{lem.I_econtainment2} and Lemma \ref{lem.PhiI_e}, for $e\gg0$ we have (note we are using Remark \ref{remark fin pd analytic spread} here):
\begin{align*}
R & \subseteq \Phi_e(F^e_*(c (I_e(Q):Q) )) & \textrm{ by Theorem \ref{thm.LiftFpure}} \\
& \subseteq  \Phi_e\left( F^e_* \left( Q^{[q]} I_e( Q^{(n)}) : \left( Q^{(n+h-1)} \right)^{[q]}\right) \right) & \textrm{ by Lemma \ref{lem.I_econtainment2}} \\
& \subseteq \left( QQ^{(n)}: Q^{(n+h-1)} \right) & \textrm{ by Lemma } \ref{lem.PhiI_e}.
\end{align*}
Thus $Q^{(n+h-1)} \subseteq Q Q^{(n)}$ for all $n \geqslant 1$. The containment $Q^{((h-1)(n-1)+1)}\subseteq Q^n$ for all $n \geqslant 1$ now follows by induction on $n$, as in \cite[Theorem 4.1]{GrifoHuneke}. \qedhere
\end{enumerate}
\end{proof}

The following corollary is immediate.

\begin{corollary}
\label{sfr finite pd n=2}
	Let $R$ be an $F$-finite Gorenstein ring and $Q \subseteq R$ be an ideal with $\pd(R/Q)<\infty$ and big height $2$. If $R/Q$ is strongly $F$-regular, then $Q^{(n)}=Q^n$ for all $n\geq 1$.
\end{corollary}


\section{Ideals of infinite projective dimension}

In this section, we prove versions of the symbolic power containments for ideals not necessarily of finite projective dimension. We begin with several lemmas.

\begin{lemma}[Hochster--Huneke, Lemma 3.6 in \cite{comparison}]
\label{lem.Jacobian}
Let $R$ be a geometrically reduced equidimensional $k$-algebra finitely generated over a field $k$ of prime characteristic $p$. Let $I$ be an ideal of $R$ and let $W$ denote the multiplicative system $R-\cup_{P\in\Ass(R/I)}P$. Let $J = J(R/k)$ be the Jacobian ideal. Then for every $q = p^e$,
$$J^{[q]} ( I^{[q]})^W \subseteq I^{[q]}.$$
\end{lemma}

Similar to Lemma \ref{lem.I_econtainment1}, we have the following:
\begin{lemma}
\label{lem.I_econtainmentInfPd 1}
	Let $R$ be a geometrically reduced equidimensional $k$-algebra finitely generated over a field $k$ of prime characteristic $p$ and let $Q\subseteq R$ be an ideal. Let $h$ denote the largest analytic spread of $QR_P$ for all $P\in\Ass(R/Q)$. Then for all $n \geqslant 1$,
	$$\left( I : Q \right) \subseteq \left( I \left( Q^{n-1} \right)^{[q]} : \left( J^{n-1} Q^{(hn-h+1)} \right)^{[q]} \right)$$
	for all $q = p^e \gg 0$.
\end{lemma}
\begin{proof}
	Fix $n \geqslant 1$. Let $t \in \left( I : Q \right)$. Then
	\[
        t \left( J^{n-1} Q^{(hn-h+1)} \right)^{[q]} \subseteq (tQ) \cdot \left( Q^{(hn-h+1)} \right)^{q-1} \left( J^{[q]} \right)^{n-1} \subseteq I \, \left( Q^{(hn-h+1)} \right)^{q-1} \left( J^{[q]} \right)^{n-1}.
    \]
    It is thus enough to show that
    $$\left( Q^{(hn-h+1)} \right)^{q-1} \left( J^{[q]} \right)^{n-1} \subseteq \left( Q^{[q]} \right)^{n-1},$$
    which will follow from the following containment:
    $$J^{[q]} \left( Q^{(hn-h+1)} \right)^{\left\lfloor \frac{q-1}{n-1} \right\rfloor} \subseteq Q^{[q]}.$$
This follows from Lemma \ref{lem.ContainmentContraction} and Lemma \ref{lem.Jacobian}.
\end{proof}

The next two lemmas will be used in the strongly $F$-regular case.
\begin{lemma}
\label{lem.ContainmentContraction SFR}
Let $Q\subseteq R$ be an ideal. Let $h\geq 2$ and $n\geq 2$ be positive integers such that $h$ is at least the largest minimal number of generators of $QR_P$ for all $P\in\Ass(R/Q)$.
Then for all $q = p^e$, we have
	$$Q^{(h-1)(n-1)} \left( Q^{((h-1)(n-1)+1)} \right)^{q-1}  \subseteq  \left((Q^{((h-1)(n-2)+1)})^{[q]}\right)^W,$$
where $W$ denotes the multiplicative system $R-\cup_{P\in\Ass(R/Q)}P$.
\end{lemma}
\begin{proof}
It suffices to prove the containment after localizing at $P$ for each $P\in\Ass(R/Q)$. Therefore we may assume that both sides are the ordinary powers of $Q$ and $Q$ is generated by $h$ elements. Since $h\geq 2$, for all $q$ we have
$$
(h-1)(n-1) + ((h-1)(n-1)+1)(q-1) = ((h-1)(n-1)+1)q-1 \geq ((h-1)(n-1)+1)q-h+1
$$
Thus we have
$$
Q^{(h-1)(n-1) + ((h-1)(n-1)+1)(q-1)}\subseteq Q^{((h-1)(n-1)+1)q-h+1}\subseteq (Q^{(h-1)(n-2)+1})^{[q]}
$$
where the last inclusion follows from the Pigeonhole Principle: if $I$ is an ideal generated by $h$ elements then $I^{Nq-h+1}\subseteq (I^{N-(h-1)})^{[q]}$.
\end{proof}

\begin{lemma}\label{lem.I_econtainmentInfPd 2}
Let $R$ be a geometrically reduced equidimensional $k$-algebra finitely generated over a field $k$ of prime characteristic $p$ and let $Q\subseteq R$ be an ideal. Let $h\geq 2$ and $n\geq 2$ be positive integers such that $h$ is at least the largest minimal number of generators of $QR_P$ for all $P\in\Ass(R/Q)$.
Then for all $q =p^e \gg0$ and all $I\subseteq R$, we have
	$$\left( Q^{(h-1)(n-1)+1} : Q^{((h-1)(n-1)+1)} \right) \left( I : Q \right) \subseteq \left( I  \left( Q^{((h-1)(n-2)+1)} \right)^{[q]} : \left( J Q^{((h-1)(n-1)+1)} \right)^{[q]} \right).$$
In particular, there exists $c$ not in any minimal prime of $Q$ such that
	$$c \left( I : Q \right) \subseteq \left( I  \left( Q^{((h-1)(n-2)+1)} \right)^{[q]} : \left( J Q^{((h-1)(n-1)+1)} \right)^{[q]} \right).$$
\end{lemma}
\begin{proof}
	Fix $n \geqslant 1$. Let $t \in \left( I : Q \right)$ and $s \in \left( Q^{(h-1)(n-1)+1} : Q^{((h-1)(n-1)+1)} \right)$. Then
	\begin{align*}
        (st) \left( J Q^{((h-1)(n-1)+1)} \right)^{[q]} & \subseteq t \left( sQ^{((h-1)(n-1)+1)} \right) \cdot \left( Q^{((h-1)(n-1)+1)} \right)^{q-1}  J^{[q]}\\
       & \subseteq (tQ) Q^{(h-1)(n-1)} \left( Q^{((h-1)(n-1)+1)} \right)^{q-1} J^{[q]}  \\
       & \subseteq I Q^{(h-1)(n-1)} \left( Q^{((h-1)(n-1)+1)} \right)^{q-1} J^{[q]}.
    \end{align*}
    It is thus enough to show that
    $$Q^{(h-1)(n-1)} \left( Q^{((h-1)(n-1)+1)} \right)^{q-1} J^{[q]} \subseteq \left( Q^{((h-1)(n-2)+1)} \right)^{[q]}.$$
This follows from
	\begin{align*}
		Q^{(h-1)(n-1)} \left( Q^{((h-1)(n-1)+1)} \right)^{q-1} J^{[q]} & \subseteq J^{[q]}  \left((Q^{((h-1)(n-2)+1)})^{[q]}\right)^W & (\textrm{By Lemma \ref{lem.ContainmentContraction SFR}}) \\
		& \subseteq \left( Q^{((h-1)(n-2)+1)} \right)^{[q]}. & (\textrm{By Lemma \ref{lem.Jacobian}})
	\end{align*}
Note that we can apply Lemma \ref{lem.Jacobian} here because $Q$ and $Q^{((h-1)(n-2)+1)}$ have the same set of associated primes.
\end{proof}

The next crucial lemma shows that any Cartier map can be lifted at the expense of multiplying by elements in the Jacobian ideal.

\begin{lemma}
\label{lem.LiftingMapJacobian}
	Let $R$ be a geometrically reduced equidimensional $k$-algebra finitely generated over an infinite field $k$ of prime characteristic $p$ and let $Q\subseteq R$ be an ideal. Let $J=J(R/k)$ denote the Jacobian ideal of $R$.
	\begin{enumerate}
		\item If $R/Q$ is $F$-pure, then for every $x\in J$ and all $e>0$, there exists $\varphi\in\Hom_R(F^e_*R, R)$ such that $x=\varphi(F^e_*1)$ mod $Q$ and that $Q\subseteq I_e^\varphi(Q)$.
		\item If $R/Q$ is strongly $F$-regular, then for every $c$ not in any minimal prime of $Q$ and for every $x\in J$, there exists $e\gg0$ and $\varphi\in\Hom_R(F^e_*R, R)$ such that $x=\varphi(F^e_*c)$ mod $Q$ and that $Q\subseteq I_e^\varphi(Q)$.
	\end{enumerate}
\end{lemma}
\begin{proof}
We first claim that $J$ annihilates $\Ext^1_R(F^e_*R, -)$ for all $e>0$. By \cite[Theorem 3.4]{comparison}, when $k$ is infinite,\footnote{Note that \cite[Theorem 3.4]{comparison} requires that we enlarge $k$ to $k(t)$, but this is only needed to guarantee that we have an infinite field so that we can pick general elements.} $J$ can be generated by elements $c$ such that there exists a polynomial ring $A$ over $k$ ($A$ depends on $c$) such that $c \cdot F^e_*R\subseteq F^e_*A\otimes_AR$. It follows that the map $F^e_*R \longrightarrow F^e_*R$ given by multiplication by $c$ factors through a free $R$-module $A_e:=F^e_*A\otimes_AR$ (since $F^e_*A$ is free over $A$ as $A$ is polynomial ring). The map on $\Ext^1_R(F^e_*R, -)$ induced by multiplication by $c$ map must then factor though $\Ext^1_R(A_e, -)=0$, thus multiplication by $c$ on $\Ext^1_R(F^e_*R, -)$ is the $0$-map. As a consequence, $J$ annihilates $\Ext^1_R(F^e_*R, -)$.

The short exact sequence
$$\xymatrix{0 \ar[r] & Q \ar[r] & R \ar[r] & R/Q \ar[r] & 0}$$
induces the exact sequence
$$\xymatrix{\Hom_R(F^e_*R, R) \ar[r] & \Hom_R(F^e_*R, R/Q) \ar[r] & \Ext^1_R(F^e_*R, Q)}.$$
Since $J \cdot \Ext^1_R(F^e_*R, Q) = 0$, every map in $\Hom_R(F^e_*R, R/Q)$ lifts to a map in $\Hom_R(F^e_*R, R)$ after multiplying by elements in $J$.

\begin{enumerate}

\item Now we fix $x\in J$. Consider the splitting $\xymatrix{F^e_*(R/Q) \ar[r]^-{\phi} & R/Q}$, $\phi(F^e_* 1) = 1$.
The composition
$$\xymatrix{F^e_*(R/Q) \ar[r]^-{\phi} & R/Q \ar[r]^-x & R/Q}$$
lifts to a map $\varphi$: $F^e_*R \longrightarrow R$, meaning, there exists a commutative diagram
$$\xymatrix{F^e_* R \ar[r] \ar[d]_-{\varphi} & F^e_*(R/Q) \ar[d]^-{x \cdot \phi(-)} \\ R \ar[r] & R/Q}$$
Tracing the image of $F^e_*1 \in F^e_* R$ through the diagram, we obtain that $x=\varphi(F^e_*1)$ mod $Q$. Moreover, since $\varphi$ induces a map $F^e_*(R/Q)\to R/Q$, we have $Q\subseteq I_e^\varphi(Q)$.

\item We fix $c$ not in any minimal prime of $Q$ and fix $x\in J$. Since $R/Q$ is strongly $F$-regular, there exists $e\gg0$ and $\phi\in \Hom_{R/Q}(F^e_*(R/Q), R/Q)$ such that $\phi(F^e_*c)=1$. Again the composition
$$\xymatrix{F^e_*(R/Q) \ar[r]^-{\phi} & R/Q \ar[r]^-x & R/Q}$$
lifts to a map $\varphi$: $F^e_*R \longrightarrow R$. We consider the commutative diagram
$$\xymatrix{F^e_* R \ar[r] \ar[d]_-{\varphi} & F^e_*(R/Q) \ar[d]^-{x \cdot \phi(-)} \\ R \ar[r] & R/Q}$$
Tracing the image of $F^e_*c \in F^e_* R$ through the diagram, we obtain that $x=\varphi(F^e_*c)$ mod $Q$. Moreover, since $\varphi$ induces a map $F^e_*(R/Q)\to R/Q$, we have $Q\subseteq I_e^\varphi(Q)$. \qedhere
\end{enumerate}
\end{proof}

We can now prove our main result on symbolic power containments for ideals not necessarily of finite projective dimension.

\begin{thm}
\label{thm.MainInfPdim}
Let $R$ be a geometrically reduced equidimensional $k$-algebra finitely generated over a field $k$ of prime characteristic $p$. Let $Q \subseteq R$ be an ideal and let $J=J(R/k)$ be the Jacobian ideal of $R$. Then
\begin{enumerate}
  \item If $R/Q$ is $F$-pure and $h=\BHt(Q)$, then $J^n Q^{(hn-h+1)} \subseteq Q^n$ for all $n \geq 1$. Moreover, if there exists a Frobenius splitting of $R/Q$ that lifts to a Frobenius splitting of $R$, then we have $J^{n-1} Q^{(hn-h+1)} \subseteq Q^n$ for all $n\geq 1$.
  \item If $R/Q$ is strongly $F$-regular and $h\geq 2$ is at least the largest minimal number of generators of $QR_P$ for all $P\in\Ass(R/Q)$, then $J^{2n-2} Q^{((h-1)(n-1)+1)} \subseteq Q^n$ for all $n \geqslant 1$. Moreover, if for every $c$ not in any minimal prime of $Q$, there exists $e\gg0$ and $\phi\in\Hom_R(F_*^e(R/Q), R/Q)$ which sends $F^e_*c$ to $1$ that lifts to $\varphi\in \Hom_R(F^e_*R, R)$, then $J^{n-1} Q^{((h-1)(n-1)+1)} \subseteq Q^n$.
\end{enumerate}
\end{thm}

\begin{proof}$\,$
	\begin{enumerate}[(1)]
	\item We can replace $R$ by $R':=R\otimes_kk(t)$ and $Q$ by $Q':=QR'$, then $R'/Q'$ is still $F$-pure, $h=\BHt(Q')$, $J(R'/k(t))=JR'$, and it is enough to check the containment in $R'$ since it is faithfully flat over $R$. Therefore we can and we will assume $k$ is infinite in order to invoke Lemma \ref{lem.LiftingMapJacobian}.

We use induction on $n$. If $n=1$ then the containment is obvious. So we assume $n\geq 2$ and we assume the containment for $n-1$, that is, $J^{n-1} Q^{(h(n-1)-h+1)} \subseteq Q^{n-1}$.

For every $x\in J$, we fix a map $\varphi\in\Hom_R(F_*^eR, R)$ as in Lemma \ref{lem.LiftingMapJacobian}. Now by Lemma \ref{lem.LiftingMapJacobian} and Lemma \ref{lem.I_econtainmentInfPd 1}, for all $q = p^e\gg0$ we have
	$$R=\left( I^\varphi_e(Q) : Q \right) \subseteq \left( I^\varphi_e(Q) \left( Q^{n-1} \right)^{[q]} : \left( J^{n-1} Q^{(hn-h+1)} \right)^{[q]} \right).$$
Here we are using that the largest analytic spread of $QR_P$ for all $P\in\Ass(R/Q)$ is the same as the big height of $Q$ because $Q$ is radical. Applying $\varphi(F^e_*-)$ and by Lemma \ref{lem.PhiI_e}, we have
	$$\varphi(F^e_*R) \subseteq \left( Q^n : J^{n-1} Q^{(hn-h+1)} \right).$$
Since $x=\varphi(F^e_*1)$ mod $Q$ by Lemma \ref{lem.LiftingMapJacobian}, we know that $x\in \left( Q^n : J^{n-1} Q^{(hn-h+1)} \right)$ mod $Q$. Since this is true for all $x\in J$, we have
	$$J \subseteq \left( Q^n : J^{n-1} Q^{(hn-h+1)} \right) + Q.$$
	By our induction hypothesis, $J^{n-1} Q^{(hn-h+1)}\subseteq J^{n-1} Q^{(h(n-1)-h+1)} \subseteq Q^{n-1}$. Therefore,
	$$Q J^{n-1} Q^{(hn-h+1)} \subseteq Q \cdot Q^{n-1} = Q^n,$$
	and thus
	$Q \subseteq \left( Q^n : J^{n-1} Q^{(hn-h+1)} \right)$.
	Therefore,
	$$J \subseteq \left( Q^n : J^{n-1} Q^{(hn-h+1)} \right),$$
	and thus
	$$J^{n} Q^{(hn-h+1)} = J \left( J^{n-1} Q^{(hn-h+1)} \right) \subseteq Q^{n}.$$
Finally, if there is a Frobenius splitting $\phi$ of $R/Q$ that lifts to a Frobenius splitting $\varphi$ of $R$, then directly applying Lemma \ref{lem.I_econtainmentInfPd 1} (note that, in this case, we do not need to enlarge $R$ to assume $k$ is infinite since we will not need Lemma \ref{lem.LiftingMapJacobian} in the following argument), we have
$$R=\left( I^\varphi_e(Q) : Q \right) \subseteq \left( I^\varphi_e(Q) \left( Q^{n-1} \right)^{[q]} : \left( J^{n-1} Q^{(hn-h+1)} \right)^{[q]} \right).$$
Applying $\varphi(F^e_*-)$ and by Lemma \ref{lem.PhiI_e} (and using that $\varphi$ is a Frobenius splitting), we know that
$$R\subseteq QQ^{n-1}:\left( J^{n-1} Q^{(hn-h+1)} \right),$$
that is, $J^{n-1} Q^{(hn-h+1)}\subseteq Q^n$.
	
	\item When $R/Q$ is strongly $F$-regular, we can replace $R$ by $R\otimes_kk(t)$ to assume that $k$ is infinite. We use induction on $n$. If $n=1$ then the containment is obvious. So we assume $n\geq 2$ and we assume the containment for $n-1$, that is, $J^{2n-4} Q^{((h-1)(n-2)+1)} \subseteq Q^{n-1}$.

For every $x\in J$, we fix a map $\varphi\in\Hom_R(F_*^eR, R)$ as in Lemma \ref{lem.LiftingMapJacobian}. Now by Lemma \ref{lem.LiftingMapJacobian} and Lemma \ref{lem.I_econtainmentInfPd 2}, for all $q = p^e\gg0$ we have
$$cR=c \left( I_e^\varphi(Q) : Q \right) \subseteq \left( I_e^\varphi(Q)  \left( Q^{((h-1)(n-2)+1)} \right)^{[q]} : \left( J Q^{((h-1)(n-1)+1)} \right)^{[q]} \right).$$
	Applying $\varphi(F^e_*-)$ and by Lemma \ref{lem.PhiI_e}, we have
	$$\varphi(F^e_*(cR)) \subseteq \left(QQ^{((h-1)(n-2)+1)}: J Q^{((h-1)(n-1)+1)} \right).$$
Since $x=\varphi(F^e_*c)$ mod $Q$ by Lemma \ref{lem.LiftingMapJacobian}, we know that $x\in \left(QQ^{((h-1)(n-2)+1)}: J Q^{((h-1)(n-1)+1)} \right)$ mod $Q$. Since this is true for all $x\in J$, we have
	$$J \subseteq \left(QQ^{((h-1)(n-2)+1)}: J Q^{((h-1)(n-1)+1)} \right) + Q.$$
Clearly, $Q\subseteq \left(QQ^{((h-1)(n-2)+1)}: J Q^{((h-1)(n-1)+1)} \right)$ because
$$QJ Q^{((h-1)(n-1)+1)}\subseteq QQ^{((h-1)(n-1)+1)}\subseteq QQ^{((h-1)(n-2)+1)}.$$
Therefore we have
$$J \subseteq \left(QQ^{((h-1)(n-2)+1)}: J Q^{((h-1)(n-1)+1)} \right),$$
and thus
$$J^2 Q^{((h-1)(n-1)+1)}\subseteq QQ^{((h-1)(n-2)+1)}.$$
Multiplying both side by $J^{2n-4}$ and use the induction hypothesis, we have
$$J^{2n-2} Q^{((h-1)(n-1)+1)} \subseteq J^{2n-4} QQ^{((h-1)(n-2)+1)}\subseteq QQ^{n-1}= Q^n.$$
Finally, if for every $c$ not in any minimal prime of $Q$, there exists $e\gg0$ and $\phi\in\Hom_R(F_*^e(R/Q), R/Q)$ which sends $F^e_*c$ to $1$ that lifts to $\varphi\in \Hom_R(F^e_*R, R)$, then directly applying Lemma \ref{lem.I_econtainmentInfPd 2} (again, in this case, we do not need to enlarge $R$ to assume $k$ is infinite since we will not need Lemma \ref{lem.LiftingMapJacobian}), for all $q=p^e\gg0$ we have
$$cR=c \left( I_e^\varphi(Q) : Q \right) \subseteq \left( I_e^\varphi(Q)  \left( Q^{((h-1)(n-2)+1)} \right)^{[q]} : \left( J Q^{((h-1)(n-1)+1)} \right)^{[q]} \right).$$
Applying $\varphi(F^e_*-)$ and by Lemma \ref{lem.PhiI_e} (and using that $\varphi$ sends $F^e_*c$ to $1$), we know that
$$R \subseteq \left(QQ^{((h-1)(n-2)+1)}: J Q^{((h-1)(n-1)+1)} \right) .$$
Therefore we have
$$J^{n-1}Q^{((h-1)(n-1)+1)}=J^{n-2}J Q^{((h-1)(n-1)+1)}\subseteq QJ^{n-2}Q^{((h-1)(n-2)+1)}\subseteq QQ^{n-1}=Q^n$$
where we have used induction on $n$. \qedhere
	\end{enumerate}
\end{proof}

\section{$F$-pure threshold results}

In this section, we prove versions of symbolic power containments in singular rings involving the $F$-pure threshold. These can be viewed as extensions of the result in \cite[Remark 3.4]{TakagiYoshida}. However, our approach is different from \cite{TakagiYoshida}, inspired by our extension of the ideas in \cite{GrifoHuneke} in the previous sections. We recall that the \emph{$F$-pure threshold of} $I$, $\fpt(I)$, is the supreme of all $t\geq 0$ such that the pair $(R,I^t)$ is $F$-pure. When $(R,\m)$ is strongly $F$-regular local, this is the same as $\lim_{e\to\infty}\frac{\max\{r | I^r \nsubseteq I_e(\m)\}}{p^e}$.

\begin{thm}\label{fpt finite pdim}
Let $R$ be a strongly $F$-regular ring. Let $I\subseteq R$ be a radical ideal with $\BHt(I)=h$ and $\pd(R/I)<\infty$. If the pair $(R, (I^{(a)})^{b})$ is strongly $F$-regular for some $a$ and $b\in\mathbb{R}_{\geq 0}$, then $I^{(hn-\lfloor ab \rfloor)}\subseteq I^n$ for all $n \geqslant 1$. In particular,
$$I^{(hn-\lfloor\fpt(I)\rfloor)}\subseteq I^n$$
for all $n \geqslant 1$.
\end{thm}

\begin{proof}
It is enough to prove the containment after localizing at each maximal ideal of $R$ (note that all the hypotheses are preserved under localization), thus we may assume $(R,\m)$ is local. Since $(R, (I^{(a)})^{b})$ is strongly $F$-regular, we know that $\fpt(I^{(a)})>b$ and thus there exists $\epsilon>0$ such that for all $q=p^e\gg0$, $(I^{(a)})^{\lfloor(b+\epsilon)q\rfloor}\nsubseteq I_e(\m)$.

Fix $n \geqslant 1$. We claim that there exists $q\gg0$ such that
\begin{equation}\label{eq fpt}
(I^{(a)})^{\lfloor(b+\epsilon)q\rfloor}\subseteq \left( I_e(I^n):(I^{(hn-\lfloor ab \rfloor)})^{[q]} \right).
\end{equation}
This will finish the proof: indeed, if $I^{(hn- \lfloor ab \rfloor)}\nsubseteq I^n$, then $\left( I^n : I^{(hn-\lfloor ab\rfloor)} \right) \subseteq \m$ which implies $I_e(I^n):(I^{(hn-\lfloor ab \rfloor)})^{[q]} \subseteq I_e(\m)$ by Lemma \ref{lem.PhiI_e}, and hence
$$(I^{(a)})^{\lfloor(b+\epsilon)q\rfloor}\subseteq \left( I_e(I^n):(I^{(hn-\lfloor ab \rfloor)})^{[q]} \right) \subseteq I_e(\m),$$
which is a contradiction.

It remains to prove (\ref{eq fpt}); since $(I^{[q]})^n\subseteq I_e(I^n)$, it is enough to show that
\[
(I^{(hn-\lfloor ab \rfloor)})^q\cdot (I^{(a)})^{\lfloor(b + \epsilon) q \rfloor} \subseteq (I^{[q]})^n,
\]
and thus enough to prove
\[
(I^{(hn-\lfloor ab \rfloor)})^{\left\lfloor\frac{q}{n}\right\rfloor}\cdot (I^{(a)})^{\left\lfloor\frac{(b+\epsilon)q}{n}\right\rfloor}\subseteq I^{[q]}.
\]
Since $\pd(R/I)<\infty$, $\Ass(R/I^{[q]})=\Ass(R/I)$. Therefore, it is enough to check this last containment after localization at the minimal primes of $I$. But since $I$ is radical and $\pd(R/I)<\infty$, localizing at minimal primes of $I$ will yield a regular local ring. Therefore we may assume $I$ is generated by $h$ elements and that its symbolic and ordinary powers coincide. We set $q=kn+r$ with $0\leq r<n$; note that $\lfloor\frac{q}{n}\rfloor=k$. When $q\gg0$, we also must have $k\gg0$, so for $q\gg0$ we may assume $\epsilon k-1 \geq \frac{h(r-1)+1}{a}$, and thus
\begin{eqnarray*}
(hn- \lfloor ab \rfloor)\left\lfloor\frac{q}{n}\right\rfloor + a \left\lfloor \frac{(b+\epsilon)q}{n}\right\rfloor & \geq & (hn-ab)\left\lfloor\frac{q}{n}\right\rfloor + a \left\lfloor \frac{(b+\epsilon)q}{n}\right\rfloor \\
{} & \geq & (hn-ab)k + a \left\lfloor(b + \epsilon) k \right\rfloor \\
  {} &\geq& (hn-ab)k+a((b+\epsilon)k-1) \\
  {} &=& (hn-ab)k+abk+a(\epsilon k-1)  \\
  {} &=& hnk+a(\epsilon k-1)\\
     &\geq& hnk+h(r-1)+1\\
     &=&  hq-h+1.
\end{eqnarray*}
Therefore, after localization at each minimal prime of $I$,
$$(I^{(hn-ab)})^{\left\lfloor\frac{q}{n}\right\rfloor}\cdot (I^{(a)})^{\left\lfloor\frac{(b +\epsilon)q}{n}\right\rfloor} \subseteq I^{hq-h+1}\subseteq I^{[q]},$$
where the last inclusion follows by the Pigeonhole Principle, since $I$ is locally generated by at most $h$ elements.
\end{proof}

Similar computations also yield a statement for ideals not necessarily of finite projective dimension, as we show below in \autoref{fpt infinite pdim}.  The referee pointed out to us an alternate and straightforward refinement of this result as a straightforward consequence  of Takagi's subadditivity formula \cite[Theorem 2.7]{TakagiFormulas}.  We provide the referee's proof below in \autoref{referee}.

\begin{thm}[{\itshape c.f.} \cite{TakagiFormulas}]\label{fpt infinite pdim}
Let $R$ be a strongly $F$-regular, geometrically reduced equidimensional $k$-algebra finitely generated over a field $k$ of prime characteristic $p$. Let $I\subseteq R$ be a radical ideal, $h$ be the largest minimal number of generators of $IR_P$ as $P$ runs through the associated primes of $I$, and let $J$ be the Jacobian ideal of $R$. If the pair $(R, (I^{(a)})^{b})$ is strongly $F$-regular for some $a$ and $b\in\mathbb{R}_{\geq 0}$, then $J^n I^{(hn-\lfloor ab \rfloor)}\subseteq I^n$ for all $n \geqslant 1$. In particular, $$J^n I^{(hn-\lfloor\fpt(I)\rfloor)}\subseteq I^n$$ for all $n \geqslant 1$.
\end{thm}

\begin{proof}
The strategy is very similar to the proof of Theorem \ref{fpt finite pdim}. It is enough to prove the containment after localizing at each maximal ideal of $R$, thus we may assume $(R,\m)$ is local. Since $(R, (I^{(a)})^{b})$ is strongly $F$-regular, $\fpt(I^{(a)})>b$ and thus there exists $\epsilon>0$ such that for all $q=p^e\gg0$, $(I^{(a)})^{\lfloor(b+\epsilon)q\rfloor}\nsubseteq I_e(\m)$.

Fix $n \geqslant 1$. We claim that there exists $q\gg0$ such that
\begin{equation}\label{need}
(I^{(a)})^{\lfloor(b+\epsilon)q\rfloor}\subseteq \left( I_e(I^n):(J^n I^{(hn-\lfloor ab \rfloor)})^{[q]} \right).
\end{equation}
Note that this claim will finish the proof: if $J^n I^{(hn-\lfloor ab \rfloor)} \nsubseteq I^n$, then $\left( I^n : J^n I^{(hn-\lfloor ab\rfloor)} \right) \subseteq \m$, which implies $\left( I_e(I^n):(J^n I^{(hn-\lfloor ab \rfloor)})^{[q]} \right) \subseteq I_e(\m)$ by Lemma \ref{lem.PhiI_e}, and hence
$$(I^{(a)})^{\lfloor(b+\epsilon)q\rfloor}\subseteq \left( I_e(I^n):(J^n I^{(hn-\lfloor ab \rfloor)})^{[q]} \right) \subseteq I_e(\m),$$
 which is a contradiction.

It remains to prove (\ref{need}); since $(I^{[q]})^n\subseteq I_e(I^n)$, it is enough to show that
\[
(J^n I^{(hn-\lfloor ab \rfloor)})^{[q]}\cdot (I^{(a)})^{\lfloor(b + \epsilon) q \rfloor} \subseteq (I^{[q]})^n,
\]
and thus enough to prove
\[
J^{[q]} (I^{(hn-\lfloor ab \rfloor)})^{\left\lfloor\frac{q}{n}\right\rfloor}\cdot (I^{(a)})^{\left\lfloor\frac{(b+\epsilon)q}{n}\right\rfloor}\subseteq I^{[q]}.
\]
By Lemma \ref{lem.Jacobian}, $J^{[q]} \left( I^{[q]} \right)^W \subseteq I^{[q]}$, where $W$ is the complement of the union of all the associated primes of $I$. Thus all we need to show is that
$$(I^{(hn-\lfloor ab \rfloor)})^{\left\lfloor\frac{q}{n}\right\rfloor} \cdot (I^{(a)})^{\left\lfloor\frac{(b+\epsilon)q}{n} \right\rfloor} \subseteq \left( I^{[q]} \right)^W,$$
and it is sufficient to prove this containment after localizing at each associated prime $P$ of $I$. By our assumption on $h$, after localization at $P$, $I$ can be generated by $h$ elements. The rest of the proof is the same as in Theorem \ref{fpt finite pdim}.
\end{proof}

\begin{remark}\label{referee}
The referee has kindly provided us with the following improvement of Theorem \ref{fpt infinite pdim}, using the subadditivity property of asymptotic test ideals \cite{TakagiFormulas}. With the same notation as in Theorem \ref{fpt infinite pdim}, since $(R, (I^{(a)})^b)$ is strongly $F$-regular, the asymptotic test ideal $\tau(\lfloor ab \rfloor \cdot I^{(\bullet)})$ is the unit ideal.\footnote{We refer to \cite{TakagiFormulas} for the precise definition of asymptotic test ideal. In our context, we are considering the graded family of ideals $\mathfrak{a}_\bullet:=\{\mathfrak{a}_n\}_n$ such that $\mathfrak{a}_n=I^{(n)}$, and we use $I^{(\bullet)}$ to abbreviate this notation.} Therefore, we have
\begin{eqnarray*}
  J^{n-1}I^{(hn-\lfloor ab \rfloor)} &=& J^{n-1}I^{(hn-\lfloor ab \rfloor)}\tau(\lfloor ab \rfloor \cdot I^{(\bullet)}) \\
   &\subseteq & J^{n-1}\tau(hn \cdot I^{(\bullet)}) \\
   &\subseteq& \tau(h\cdot I^{(\bullet)})^n \subseteq I^n,
\end{eqnarray*}
where the first inclusion follows from \cite[Lemma 4.5]{TakagiFormulas}, the second inclusion follows from \cite[Proposition 4.4]{TakagiFormulas}, and the last inclusion follows since $\tau(h\cdot I^{(\bullet)}) \subseteq I$: after localizing at each associated prime $P$ of $I$, we have $\tau(h\cdot I^{(\bullet)}R_P) = \tau(I^hR_P) \subseteq IR_P$ by \cite[Theorem 1.11]{TakagiFormulas} (in fact, here we only need to assume that $h$ is the largest analytic spread of $IR_P$ as $P$ runs over all associated primes of $I$).
\end{remark}

\section{Examples}

In this section, we collect a few examples: an example suggesting Theorem \ref{thm.MainInfPdim} can be improved in the strongly $F$-regular case, and examples of ideals of finite projective dimension for which the conclusion of Theorem \ref{thm.MainFinitePdim} does not follow from previously known results in the regular setting.

\begin{example}\label{example xy-zk}

	Let $R = K [ x,y,z ]/(xy-z^k)$ where $k \geqslant 2$ and $K$ is a field of characteristic $p \nmid k$. Consider the prime ideal $Q = (x,z)$ in $R$, which has infinite projective dimension. Theorem \ref{thm.MainFinitePdim} does not apply to $Q$; if it did, it would say that $Q^{(n)} = Q^n$ for all $n \geqslant 1$, since $R/Q$ is strongly $F$-regular. In fact, $Q^{(n)} \neq Q^n$ for all $n \geqslant 2$, since $x^{n+r} \in Q^{(kn+r)} \setminus Q^{kn}$ for all $n \geqslant 1$ and $0 \leqslant r < k$. It turns out that $Q^{(kn)} = \left( x^n \right)$. Theorem \ref{thm.MainInfPdim} does apply, and it says that $J^{n} Q^{(n)} \subseteq Q^n$ for all $n$. We claim, however, that this can be improved. First, let's show that
	$$J^{(k-1)n} Q^{(kn)} \subseteq Q^{kn} \textrm{ for all } n \geqslant 1.$$
	To do that, we start by noting that the Jacobian ideal is given by $J = \left( x,y,z^{k-1} \right)$, so $J^{(k-1)n}$ is generated by all the elements of the form $x^ay^bz^{(k-1)c}$ with $a+b+c=(k-1)n$. If $n + a - b \geqslant 0$, then
	$$\left( x^ay^bz^{(k-1)c} \right) x^n = x^{n+a-b} \left( xy \right)^b z^{(k-1)c} = x^{n+a-b} z^{kb+(k-1)c} \in Q^{n+(k-1)(b+c) + a} \subseteq Q^{n + a + b + c} = Q^{kn}.$$
	On the other hand, if $n + a - b < 0$, then we necessarily have $b > n$, so
	$$y^b \cdot x^n = y^{b-n} (xy)^{n} = y^{b-n} z^{kn} \in Q^{kn}.$$
	This proves our claim that $J^{(k-1)n} Q^{(kn)} \subseteq Q^{kn}$. We now observe that, in fact, this can be extended to
	$$J^{\left\lfloor \frac{k-1}{k} n \right\rfloor} Q^{(n)} \subseteq Q^n \textrm{ for all } n \geqslant 1.$$
	To see this, we will first need to check that for each $0 < r < k$,
	$$Q^{(kn+r)} = Q^{(k(n+1))} + \sum_{a=0}^n Q^{(k(n-a))} Q^{ak+r} = (x^{n+1}) + \sum_{a=0}^n (x^{n-a}) Q^{ak+r}.$$
	Equivalently, the statement above says that the symbolic Rees algebra of $Q$ is generated by $Q$ and $Q^{(k)} = (x)$. To show this, we need to do calculate the symbolic powers of $Q$. Since $Q$ is a quasi-homogeneous prime, its symbolic powers are generated by monomials in $x$, $y$, and $z$, and for height reasons we must have $Q^{(n)} = Q^n : (x,y,z)^\infty$. A given $f$ is in $Q^{(n)}$ if and only $f \cdot (x,y,z)^N \subseteq Q^{n}$ for some $N$, which is equivalent to checking whether $fy^N \in Q^{n}$ for some $N$. When $N$ is sufficiently large,
	$$x^ay^bz^c \cdot y^N = y^{b+N-a}(xy)^a (z^c) = y^{b+N-a}z^{ka+c} \in Q^{ka+c},$$
	so we have $x^ay^bz^c \in Q^{(n)}$ if and only if $ka + c \geqslant n$. This shows that indeed we have
	$$Q^{(kn)} = \left( Q^{(k)} \right)^n = (x^n) \textrm{ and } Q^{(kn+r)} = Q^{(k(n+1))} + \sum_{a=0}^n Q^{(k(n-a))} Q^{ak+r} \textrm{ for all } n \geqslant 1 \textrm{ and } 0 < r < k.$$
	We still need to show that $J^{\left\lfloor \frac{k-1}{k} (kn+r) \right\rfloor} Q^{(kn+r)} \subseteq Q^{kn+r}$ for all $n \geqslant 1$ and $0 < r < k$. Now
	$$\left\lfloor \frac{k-1}{k} (kn+r) \right\rfloor = \left\lfloor (k-1)n + r - \frac{r}{k} \right\rfloor = (k-1)n + r -1,$$
	and $r \geqslant 1$, so using that $J^{(k-1)n} Q^{(kn)} \subseteq Q^{kn}$ we get
	\begin{align*}
		J^{\left\lfloor \frac{k-1}{k} (kn+r) \right\rfloor} Q^{(kn+r)} & = J^{(k-1)n + r -1} Q^{(k(n+1))} + J^{(k-1)n + r -1} \sum_{a=0}^n  Q^{(k(n-a))} Q^{ak+r} \\
		& \subseteq J^{(k-1)n} Q^{(kn)}  J^{r-1}Q^{(k)} + J^{(k-1)n} Q^{(kn)} Q^r \\
		& \subseteq Q^{kn} J^{r-1}Q^{(k)} + Q^{kn} Q^r.
	\end{align*}
	So our claim is now reduced to checking that $J^{r-1} Q^{(k)} \subseteq Q^r$. Given a generator for $J^{r-1}$, say $x^ay^bz^{(k-1)c}$ with $a+b+c\geqslant r-1$, we need to check that $x^{a+1}y^bz^{(k-1)c} \in Q^r$. And indeed, either $b = 0$, in which case $a+1+(k-1)c \geqslant r$ and $x^{a+1}z^{(k-1)c} \in Q^r$, or $b \geqslant 1$ and $x^{a+1}y^bz^{(k-1)c} \in (xy) = (z^k) \subseteq Q^k \subseteq Q^{r}$.
		
	Finally, note that $x^{(k-1)n-1} \cdot x^n \notin Q^{kn}$ and thus $J^{(k-1)n-1} Q^{(kn)} \nsubseteq Q^{kn}$. We conclude that Theorem \ref{thm.MainInfPdim} can be improved for this example: the theorem says that $J^{n} Q^{(n)} \subseteq Q^n$ for all $n \geqslant 1$, while our computations show that the minimum $\alpha$ such that
	$$J^{\left\lfloor \alpha n \right\rfloor} Q^{(n)} \subseteq Q^n \textrm{ for all } n \geqslant 1$$
	is $\alpha = \frac{k-1}{k}$.
\end{example}

It is natural to ask if Theorem \ref{thm.MainInfPdim} (2) can be improved in general:

\begin{question}
	Let $R$ be a geometrically reduced equidimensional $k$-algebra, finitely generated over a field $k$ of prime characteristic $p$. Let $Q \subseteq R$ be an ideal and let $J=J(R/k)$ be the Jacobian ideal of $R$. Supoose that $R/Q$ is strongly $F$-regular and $h\geq 2$ is at least the largest minimal number of generators of $QR_P$ for all $P\in\Ass(R/Q)$. Is
	$$J^{n} Q^{((h-1)(n-1)+1)} \subseteq Q^n$$
	for all $n \geqslant 1$?
\end{question}

Given Example \ref{example xy-zk}, a positive answer to this question would essentially give an optimal value for the power of $J$ required to fix the failure of Theorem \ref{thm.MainFinitePdim} when the ideals are not necessarily of finite projective dimension.

We are grateful to Bernd Ulrich for communicating the following example and proposition to us.

\begin{example}
\label{example Bernd}
Let $S=K(u_{ij}, v_{ij})\llbracket x_1, x_2, x_3, x_4, x_5, x_6 \rrbracket$, $I=(\Delta_1, \Delta_2, \Delta_3)$ where the $\Delta$'s are the $2\times 2$ minors of the $2\times 3$ matrix $M$ whose entries are linear combinations of the $x_i$'s using $u_{ij}$, i.e., the entries in the matrix are $\sum_j u_{ij}x_j$ (so they are all generic). Let $f=\sum v_{ij}x_ix_j$ be a generic polynomial of degree $2$ and set $R=S/f$. Then $0\to R^2\xrightarrow{M} R^3\xrightarrow{(\Delta_1, \Delta_2, \Delta_3)} R\to R/I\to 0$ is a projective resolution of $R/I$ over $R$ since the rank and depth conditions are easily verified. Hence $\pd(R/I)=\Ht(I)=2$ and one can check that $R/I$ is strongly $F$-regular (we can also choose more specific matrix and $f$ to make this true, but the generic choice always works). Theorem \ref{thm.MainFinitePdim} then implies that $I^{(n)}=I^n$ for all $n$.
\end{example}

\begin{remark}
We want to point out that, in the above example, one cannot find a regular local ring $(A,\m)$ and an ideal $J\subseteq A$ with a faithfully flat extension $A\to R$ such that $I=JR$ (i.e., the pair $(R, I)$ does not come from a faithfully flat base change from a pair $(A, J)$ with $A$ regular). If so, then since $I=JR$ we know that $\Ht(J)=2$ so a minimal resolution of $A/J$ over $A$ looks like $0\to A^{n-1}\xrightarrow{M'} A^{n}\to A\to A/J\to 0$. Since $A\to R$ is faithfully flat, tensoring this resolution with $R$ we should get a minimal resolution of $R/I$. This implies that $n=3$ and that $\text{Fitt}_2(I)$ is generated by the entries in $M'$, in particular $\text{Fitt}_2(I)\subseteq \m R$. However, by the way we construct $I$, we know that $\text{Fitt}_2(I)=(x_1,\dots,x_6)$ is the maximal ideal of $R$. Therefore $(A,\m)\to (R,\n)$ is a faithfully flat local extension such that $\m R= \n$, which implies $R$ is regular, a contradiction. In particular, in Example \ref{example Bernd}, the fact $I^{(n)}=I^n$ for all $n$ does not follow from results in \cite{GrifoHuneke}.
\end{remark}

We next prove a proposition which gives another explanation for Example \ref{example Bernd}.

\begin{proposition}
\label{prop.Bernd}
Let $(R,\m)$ be a Cohen-Macaulay local or standard graded ring with an infinite residule field such that $\dim R=d$. Let $M$ be an $n\times (n+1)$ matrix of generic linear forms (i.e., elements in $\m-\m^2$) and set $I$ be the ideal generated by $n\times n$ minors of $M$ (so $I$ is a height two perfect ideal of $R$). Then
\begin{enumerate}
  \item If $d>n+1$, then $I^{(j)}=I^{j}$ for all $j$.
  \item If $d=n+1$, then $I^{(j)}\neq I^{j}$ for all $j\geq n$.
\end{enumerate}
\end{proposition}
\begin{proof} From the resolution: $0\to R^n\xrightarrow{M}R^{n+1}\to R\to R/I\to 0$, we form the complex:
$$0\to \bigwedge^j(R^n)\to \bigwedge^{j-1}(R^n)\otimes S_1(R^{n+1})\to\cdots\to\bigwedge^1(R^n)\otimes S_{j-1}(R^{n+1})\to S_j(R^{n+1})\to R\to 0.$$
This complex has abutment $R/I^j$, with length $j+1$ when $j\leq n$ and length $n+1$ when $j\geq n$. In particular, if $d\geq n+1$, it is easily verified that the rank and depth conditions are satisfied (since we use generic elements) thus the complex is a projective resolution of $R/I^j$ for all $j$.

Now if $d>n+1$, then $\depth(R/I^j)=d-\pd(R/I^j)>0$ so $H_\m^0(R/I^j)=0$. To show $I^{(j)}=I^{j}$ for all $j$, note that we may assume $I^{(j)}R_P=I^{j}R_P$ for all $P\neq \m$ by induction, but then $H_\m^0(R/I^j)=0$ implies $I^{(j)}=I^{j}$. Finally, if $d=n+1$, then $\depth(R/I^j)=d-\pd(R/I^j)=0$ for $j\geq n$, thus $0\neq H_\m^0(R/I^j)\subseteq I^{(j)}/I^j$.
\end{proof}

\begin{remark}
It is easy to see that Example \ref{example Bernd} satisfies condition $(1)$ of Proposition \ref{prop.Bernd}. Moreover, Theorem \ref{thm.MainFinitePdim} shows that under condition $(2)$ of Proposition \ref{prop.Bernd}, $R/I$ cannot be strongly $F$-regular if $R$ is Gorenstein. Here we give another explanation under the graded setup of Proposition \ref{prop.Bernd}: in fact, if $R$ is standard graded, then from the exact sequence $0\to R^n(-n-1)\to R^{n+1}(-n)\to R\to R/I\to 0$ it is easy to see that $a(R/I)=a(R)+n+1$. Since $R$ is Cohen-Macaulay of dimension $d$, $a(R)\geq -d$ and hence $a(R/I)\geq n+1-d=0$ under condition (2), thus $R/I$ cannot be strongly $F$-regular.
\end{remark}

\section*{Acknowledgements}

We thank Bernd Ulrich for giving us Example \ref{example Bernd} and Proposition \ref{prop.Bernd}. We thank Srikanth Iyengar for very helpful discussions on derived categories and for providing us an alternative approach to Lemma \ref{lem.LiftMap} in Remark \ref{rmk.Sri}. We also thank Lucho Avramov, Jack Jeffries, Claudia Miller, and Alexandra Seceleanu for valuable conversations and for giving us several interesting examples. We also thank Javier Carvajal-Rojas, Craig Huneke, Thomas Polstra, and Axel St\"{a}bler for useful conversations and comments. The first author thanks the University of Utah, where she was visiting when part of this project was completed, for their hospitality. Finally, we thank the anonymous referee for their valuable comments, especially Remark \ref{referee}.

\bibliographystyle{alpha}
\newcommand{\etalchar}[1]{$^{#1}$}

\vspace{1em}

\end{document}